\documentclass[oneside,english]{amsart}
\usepackage[T1]{fontenc}
\usepackage[latin9]{inputenc}
\usepackage{amsmath,amsthm,amssymb,latexsym,epic,eepic,bbm,enumerate}
\usepackage{babel}
\usepackage[all]{xypic}
\usepackage[parfill]{parskip}

\theoremstyle{plain}
\newtheorem{thm}{Theorem}[section]
\newtheorem{prop}[thm]{Proposition}
\newtheorem{lem}[thm]{Lemma}
\newtheorem{cor}[thm]{Corollary}
\theoremstyle{definition}
\newtheorem{defn}[thm]{Definition}
\newtheorem*{exmp}{Example}
\theoremstyle{remark}
\newtheorem{rem}[thm]{Remark}

\newcommand{\set}[1]{\ensuremath{\{#1\}}}
\newcommand{\w}{\mathbf{w}}
\newcommand{\x}{\mathbf{x}}
\newcommand{\y}{\mathbf{y}}

\newcommand{\m}{\mathfrak{m}}

\newcommand{\siso}{\mathcal{S}}
\newcommand{\la}{\leftarrow}
\newcommand{\ra}{\rightarrow}
\newcommand{\G}{\Gamma}

\makeatletter
\numberwithin{equation}{section}
\numberwithin{figure}{section}

\makeatother

\begin{document}

\author{Love Forsberg}
\title{Effective representations of Hecke-Kiselman monoids of type $A_n$}

\begin{abstract}
We prove effectiveness of certain representations of Hecke-Kiselman
monoids of type $A_n$ constructed by Ganyushkin and Mazorchuk and also 
construct further classes of effective representations for these monoids. 
As a consequence the effective dimension of  monoids of type $A_n$ is determined. 
We also show that odd Fibonacci numbers $F_{2n+1}$ appear as the cardinality
of certain bipartite HK-monoids and count the number of multiplicity
free elements in any HK-monoid of type $A_n$.
\end{abstract}
\maketitle
 
\section{Introduction}\label{s1}

Let $S$ be a monoid and $\mathcal{R}$ an integral commutative domain. A (finite dimensional) linear representation of $S$  over $\mathcal{R}$ is a homomorphism $\varphi$ from $S$ to the semigroup 
$\mathrm{Mat}_{n\times n}(\mathcal{R})$ of $n\times n$ matrices over $\mathcal{R}$. The representation $\varphi$ is called {\em effective} if different elements of $S$ are represented by different matrices. Note that a faithful representation of the semigroup algebra $\mathcal{R} S$ induces an effective representation of $S$, but the converse is false in general. The least $n\in\mathbb{N}$ such that $S$ has an effective representation in $\mathrm{Mat}_{n\times n}(\mathcal{R})$ is called the \emph{effective dimension} (of $S$ over $\mathcal{R}$)
and is denoted by $\mathrm{eff}.\dim_\mathcal{R}(S)$. For finite semigroups the problem of determining $\mathrm{eff}.\dim_\mathcal{R}(S)$ is effectively computable when $\mathcal{R}$ is an algebraically closed or a real closed field, see \cite{MS12}, but giving the answer as a closed formula is usually very hard.

C.~Kiselman defined in \cite{Kis02} a monoid generated by three operators $c,l,m$ with origins in 
convexity theory and showed that it has presentation
	\begin{multline}
		K=\langle c,l,m : c^2=c, l^2=l, m^2=m,\\
			clc=lcl=lc, cmc=mcm=mc, lml=mlm=ml\rangle.
	\end{multline}
In \cite{Kis02} it was shown that $K$ has $2^3$ idempotents and that $K$ has an effective representation
by (non-negative) integer valued $3\times 3$-matrices.

O.~Ganyushkin and V.~Mazorchuk generalized $K$ to a series of monoids $K_n$, called \emph{Kiselman monoids} (unpublished) given by the following presentation:
	\[K_n=\langle c_1,c_2,\cdots,c_n: c_i^2=c_i~ \forall i,
			c_ic_jc_i=c_jc_ic_j=c_ic_j ~\forall i\leq j \rangle.\]
It was shown by G.~Kudryavtseva and V.~Mazorchuk in \cite{KM09} that $K_n$ is a finite monoid
with $2^n$ idempotents and that $K_n$ has an effective representation by (non-negative) integer valued 
$n\times n$-matrices. The proof of the latter fact is technically rather involved.

Let $\G=(V(\G),E(\G))$ be a graph representing a disjoint union of simply laced Dynkin diagrams, 
$W_{\G}$ the corresponding Weyl group and $\mathcal{H}_q(W_\G)$, where $q\in\mathcal{R}$, 
the Hecke algebra of $W_\G$. By specializing $q=0$ we obtain an algebra which is isomorphic to the monoid 
algebra of the so-called {\em $0$-Hecke monoid} $\mathcal{H}_\G$ which has the following presentation:
	\begin{multline}
		\mathcal{H}_\G=\langle v\in V(\G): v^2=v~\forall v,\\
			vwv=wvw \text{ for all } \{v,w\}\in E(\G), vw=wv\text{ for all } \{v,w\}\not\in E(\G)\rangle.
	\end{multline}
	
O.~Ganyushkin and V.~Mazorchuk proposed in  \cite{GM11} a common generalization for $K_n$ and $\mathcal{H}_\G$ 
by introducing the so-called {\em Hecke-Kiselman monoids} which are defined as follows.

	\begin{defn}
		Let $\G=\big(V(\G),E(\G)\big)$ be a simple directed graph, i.e. a directed graph with no loops and at most one edge $x\to y$ for every ordered pair $(x,y)$. The \emph{Hecke-Kiselman monoid} $HK_\G$ of $\G$ is the quotient of the free monoid $\big(V(\G)\big)^*$ by the following relations:
		\begin{enumerate}[(I)]
			\item $x^2=x$ for all $x\in V(\G)$.
			\item If there is no edge between $x$ and $y$, then $xy=yx$.
			\item If there is an edge from $x$ to $y$ but no edge from $y$ to $x$, then $xyx=yxy=xy$.
			\item If there are edges in both directions between $x$ and $y$, then $xyx=yxy$.
		\end{enumerate}
		We refer to the above relations (including $x^2=x$) as \emph{edge relations}. Since there is no risk of confusion we will use $x\in\G$ as a shorthand for $x\in V(\G)$. A pair of vertices satisfying (II) is said to be \emph{non-adjacent}.
	\end{defn}

It is often convenient to think of a pair of edges $x\to y$ and $y\to x$ in $\G$ as a single undirected edge. 
In order to simplify statements this identification is done for the rest of the paper.
The $0$-Hecke monoid 
is recovered as a Hecke-Kiselman monoid by using the same (simply laced) Dynkin diagram and treating the edges as undirected edges. The Kiselman monoid $K_n$ corresponds
to the directed graph $\kappa_n$ with vertices $V(\kappa_n)=\{1,2,\cdots,n\}$ and edges $E(\kappa_n)=\{(i,j)|1\leq i<j\leq n\}$. We call this graph the \emph{Kiselman graph}. 
O.~Ganyushkin and V.~Mazorchuk proved that the semigroups $HK_\G$ and $HK_{\G'}$ are isomorphic if and only if
$\G$ and $\G'$ are isomorphic graphs \cite{GM11}. 

Mixed graphs and the corresponding Hecke-Kiselman monoids have been found to be a suitable mileau to study computer simulations which discretize continuous dynamical systems, via so-called sequential dynamical systems \cite{CA15}\footnote{The paper by Collina and D'Andrea cites this paper, due to a long period where  a draft of this paper was available arXiv. However, these citations are not needed for the strictly mathematical content, but only as a guide to the reader.}.  
	
\begin{defn}
		A simple directed graph $\G$ is said to be of \emph{type $A_n$} if its underlying undirected graph is the Dynkin diagram $A_n$ for some $n\in\mathbb{N}$. A graph of type $A_n$ with exactly one sink and one source is called \emph{linearly ordered}. The special case when sink and source coincide can only happen when $n=1$. A \emph{canonical order} on the vertices of a graph of type $A_n$ is one where neighboring vertices have indices that differ by 1. There are two canonical orders on a graph of type $A_n$ ($n\geq 2$). If the graph is linearly ordered, we additionally ask that the canonical order is from the source to the sink, i.e.
		\[\xymatrix{v_1\ar[r]&v_2\ar[r]&\cdots\ar[r]&v_n}.\]
	\end{defn}
	
We say that $HK_\G$ is of type $A_n$ if $\G$ is of type $A_n$. The monoids $HK_\G$ of type $A_n$
naturally appear as monoids of projection functors as defined by A.-L.~Grensing in \cite{Gre12,Paa11},
see also some further development in \cite{GrMa}. 

In view of the above the following questions arise naturally:
	\begin{enumerate}
		\item Does $HK_\G$ have $2^n$ idempotents, where $n=|V(\G)|$ is the number of vertices? If not, 
		 how many?
		\item Does $HK_\G$ have an effective representation by $n\times n$- matrices over 
		$\mathbb{Z}$?
		\item For which graphs $\G$ is the Hecke-Kiselman monoid $HK_\G$ finite? Can we calculate its cardinality, either explicitly or with some algorithm?
	\end{enumerate}
	
The second question seems to be a question about one ring, but a representation with only integer entries works universally for all rings of characteristic 0.  
	
These are the questions which we address in the present paper for various families of graphs. 
After the first version of the present paper appeared, 
R.~Aragona and A.~D'Andrea addressed the cardinality problem in the case $\Gamma$ is small (has at most four vertices) 
and discovered a nontrivial example of a graph with an unoriented edge for which the cardinality of the Hecke-Kiselman monoid is infinite, see \cite{AA13}. In this paper we study only graphs for which there are no unoriented edges,
or, equivalently, there is at most one oriented edge between any pair of vertices.

The paper is organized as follows: In Section~\ref{s2} we collected some basic notation and definitions 
from the combinatorics of words. Section~\ref{s3} contains some preliminary results on linear representations of 
some Hecke-Kiselman monoids. In Section~\ref{s4} we discuss effective representations and cardinalities of the Hecke-Kiselman monoids considered in Section~\ref{s3}. In Section~\ref{s5} we investigate obstructions to generalize our methods to further classes of  Hecke-Kiselman monoids.

All representations in this paper are linear. We set $\underline{n}=\{1,2,\cdots,n\}$.

\section{Words, content and canonical projections}\label{s2}

If $\G$ is a disjoint union of $\G_1$ and $\G_2$, then it is easy to see that 
\[HK_\G=HK_{\G_1}\oplus HK_{\G_2}.\] Thus we only need to study connected graphs.

We will use the bold font to denote a word $\w\in\big(V(\G)\big)^*$. Elements in $HK_\G$ are equivalence classes of words and are denoted by brackets: $[\w]\in HK_\G$. The empty word is denoted by $\varepsilon$.
We will make use of two binary relations on $\big(V(\G)\big)^*$: 
	\[\w\sim\w':\iff[\w]=[\w'],\text{ and}\]
	\[\w\approx\w':\iff\big(\w=\mathbf{xyz}\text{ and }\w'=\mathbf{xy'z},\text{ where }\mathbf{y}=\mathbf{y'}\text{ is an edge relation or }\w=\w'\big)\]
Note that $\sim$ is an equivalence relation while  $\approx$ is not. Moreover, $\sim$ is the transitive closure of $\approx$. Hence the statement $\w\approx\w'$ is stronger than the statement $\w\sim\w'$.

For each word $\w$ we may define the {\em content} $\mathfrak{c}(\w)\subset V(\G)$ as the set of vertices that appear at least once in $\w$. Note that the equality $\mathfrak{c}(\w)=\mathfrak{c}(\w')$ holds for each of the edge relations $\w=\w'$. Since $\mathfrak{c}(\mathbf{xyz})=\mathfrak{c}(\mathbf{x})\cup\mathfrak{c}(\mathbf{y})\cup\mathfrak{c}(\mathbf{z})$, this implies $\mathfrak{c}(\w)=\mathfrak{c}(\w')$ when $\w\approx\w'$, and, by transitivity, when $\w\sim\w'$. Thus we may define the content $\mathfrak{c}([\w])$ for each element in $[\w]\in HK_\G$.
Note that $\mathfrak{c}([\w][\w'])=\mathfrak{c}([\w])\cup\mathfrak{c}([\w'])$. In particular $\mathfrak{c}([\w])=\mathfrak{c}([\w'])$ implies $\mathfrak{c}([\w])=\mathfrak{c}([\mathbf{ww'}])$, which in turn implies that the set of elements with a fixed content form a subsemigroup. Thus $HK_\G$ is $\Lambda$-graded, where $\Lambda$ is the lattice consisting of subsets of $V(\G)$ and with usual join and meet.
	
Let $\G'\subset\G$ be a subgraph. Then $HK_{\G'}$ is a quotient of $HK_\G$ in the natural way. If $\G'$ is a full subgraph, then $HK_{\G'}$ is a submonoid of $HK_\G$, and the function defined by sending each $x\in V(\G')$ to itself and each $x\in V(\G\setminus\G')$ to $\varepsilon$ 
extends uniquely to a homomorphism $p:HK_\G\to HK_{\G'}$ which is called the {\em canonical projection onto $HK_{\G'}$}. This slight misuse of the word \emph{projection} is justified by the fact that it linearizes to a projection in the semigroup algebras.

\section{Preliminary results}\label{s3}

Let $\mathcal{R}$ be an integral domain and $W=\bigoplus_{v\in V(\G)}\mathcal{R}v$ the formal vector space
over $\mathcal{R}$ with basis $V(\G)$. Let $f:E(\G)\to \mathcal{R}\setminus\{0\}$ be a function (we will
call it a {\em weight} function). We denote the weight on the edge 
from $x$ to $y$ by $f_{xy}$ (assuming it exists). For arbitrary vertices $x$ and $y$, set
	\[\theta_x^f(y)=\begin{cases}
		y, & x\neq y;\\
		\sum_{z\to x}f_{zx}z,& x=y;
	\end{cases}\]
and extend this by linearity to an endomorphism of $W$. The empty sum, which happens when $x=y$ is source, is defined to be 0, as usual. Define the map $R_f:V(\G)\to End_\mathcal{R}(W)$ by $x\mapsto\theta_x^f$. This uniquely extends to a homomorphism $R_f:\big(V(\G)\big)^*\to End_\mathcal{R}(W)$
using the fact that $V(\G)$ is a set of free generators of $\big(V(\G)\big)^*$. The endomorphisms $\theta_x^f$ 
will be called \emph{atomic}. If the weight function is fixed in advance, we sometimes omit 
it from the notation and write simply $\theta_x$ for $\theta_x^f$. This construction generalizes the
construction of ``linear integral representations'' in \cite{GM11}, and the following statement generalizes 
\cite[Proposition~7]{GM11}.

\begin{thm}
$R_f$ induces a well-defined homomorphism $HK_\G\to End_\mathcal{R}(W)$.
\end{thm}	

\begin{proof}
We have to check that $R_f$ respects the edge relations. 
We start with the relation $x^2=x$. Since $\G$ is simple, an edge $z\to x$ implies $z\neq x$. Thus 
	\[\theta_x\circ\theta_x(y)=\begin{cases}
		y, & x\neq y;\\
		\sum_{z\to x}f_{zx}\theta_x(z), & x=y \end{cases}=
		\begin{cases}
		y, & x\neq y;\\
		\sum_{z\to x}f_{zx}z, & x=y \end{cases}=
	\theta_x(y).\]

		Relation $xy=yx$ for $x$ and $y$ different and non-adjacent. 
		\[\theta_x\circ\theta_y(z)=\begin{cases}
			\theta_x(z), & y\neq z;\\
			\sum_{\omega\to y}f_{\omega y}\theta_x(\omega), & y=z\end{cases} =
			\begin{cases}
			z, & x,y\neq z;\\
			\sum_{\omega\to x}f_{\omega x}\omega, & x=z; \\
			\sum_{\omega\to y}f_{\omega y}\omega, & y=z. 
			\end{cases}\]
		As the right hand side is symmetric in $x$ and $y$, so is the left hand side.

		Relation $xyx=yxy=xy$ when there is a directed edge from $x$ to $y$.
		\[\theta_x\circ\theta_y\circ\theta_x(z)=
		\begin{cases}
		\theta_x\circ\theta_y(z), & x\neq z;\\
		\sum_{\omega\to x}f_{\omega x}\theta_x\circ \theta_y(\omega), & x=z
		\end{cases}
		=\begin{cases}
		z, & x,y\neq z;\\
		\sum_{\omega\to y}f_{\omega y}\theta_x(\omega), & y=z;\\
		\sum_{\omega\to x}f_{\omega x}\omega, & x=z.
		\end{cases}\]
		Since there is a directed edge from $x$ to $y$ we have
		\[\sum_{\omega\to y}f_{\omega y}\theta_x(\omega)=\sum_{\omega\to y,x\neq \omega}f_{\omega y}\omega+f_{xy}\sum_{\omega \to x}f_{\omega x}\omega.\]
		On the other hand
		\[\theta_x\circ\theta_y(z)=
		\begin{cases}
		\theta_x(z), & y\neq z;\\
		\sum_{\omega\to y}f_{\omega y}\theta_x(\omega), & y=z\end{cases}
		=\begin{cases}
		z,& x,y\neq z;\\
		\sum_{\omega\to x}f_{\omega x}\omega, & x=z;\\
		\sum_{\omega\to y}f_{\omega y}\theta_x(\omega), & y=z.\end{cases}\]
		This implies $xyx=xy$. Observe that $\theta_x\circ\theta_y(z)$ has no $y$ component for any $z$. Thus, by definition, $\theta_y$ acts as identity on $\theta_x\circ\theta_y(z)$ and we get $xy=yxy$.
	\end{proof}
	
	In the case $\G$ is linearly ordered of type $A_n$, $f$ is constantly equal to $1$ and $\mathbb{Z}\subset\mathcal{R}$ we have that $R_f$ is effective, as proved in \cite{GM11}. We denote a constant function $E(\G)\to\mathcal{R}$ with value $c$ simply by $c\in\mathcal{R}$ and the corresponding representation by $R_c$.
	
\begin{lem}
Let $\G$ be linearly ordered of type $A_n$. Then any choice of the function $f$ gives an effective representation $R_f$ of $HK_\G$.
\end{lem}
	
It is important to recall that all values of $f$ are non-zero by definition. 
	
\begin{proof}
For a linearly ordered $\G$ of type $A_n$ each vertex of $\G$ is a target of at most one arrow.
Therefore the definition of $R_f$ implies that for any $v,x\in V(\G)$ the linear transformation 
$R_f([v])$ maps $x$ to a scalar multiple of
some other vertex, say $y$. By induction on the length of a word it follows that
for any $x\in V(\G)$ and $[\w]\in HK_\G$ there exists $v\in V(\G)$ and $c_{x,f,\w}\in\mathcal{R}$ 
such that $R_f([\w])(x)=c_{x,f,\w}y$. Certainly, $c_{x,f,\w}$ depends on $f$. However, we claim that 
\begin{equation}\label{eq1}
 c_{x,f,\w}\neq 0\quad\text{ implies that }\quad c_{x,f',\w}\neq 0\quad \text{ for any other }\quad f',
\end{equation}
or, in other words, that the fact that $c_{x,f,\w}$ is non-zero does not depend on $f$.
Claim  \eqref{eq1} and  effectiveness of $R_1$ established in \cite{GM11} imply the claim of our lemma.

To prove claim \eqref{eq1} assume that $c_{x,f,\w}\neq 0$. Let $\w=w_kw_{k-1}\cdots w_1$ and set $y_0=x$.
For $i=1,\dots,k$ define recursively $y_i$ as the unique vertex of $\Gamma$ such that 
$R_f([w_iw_{i-1}\cdots w_1])(x)=c_i y$ for some non-zero $c_i\in\mathcal{R}$. This is well-defined
as $c_{x,f,\w}\neq 0$. We have $c_k=c_{x,f,\w}\neq 0$.
The definition of $R_f$ and the fact that $\mathcal{R}$ is a domain imply that for $f'$ we will have that 
$R_{f'}([w_iw_{i-1}\cdots w_1])(x)=c'_i y$ for some non-zero $c'_i\in\mathcal{R}$.
In particular, $c'_k\neq 0$ and the claim follows.
\end{proof}

\begin{rem}
We note that the representation $R_f$ is not effective if $\G$ contains the subgraph
\[\xymatrix{u\ar[r]&v\ar@{-}[r]&w}\] 
To simplify notation, let 
\[A=\sum_{z\to v,z\neq w}f_{zv}z\text{ and }B=\sum_{z\to w,z\neq v}f_{zw}z.\]
Then $\theta_v(A)=\theta_w(A)=A$ and $\theta_v(B)=\theta_w(B)=B$.
From the definition of $R_f$ we have
\begin{multline*}\theta_v\circ\theta_w\circ\theta_v(v)=\theta_v\circ\theta_w (f_{wv}w+A)=
A+f_{wv}\theta_v\circ\theta_w (w)=\\
A+f_{wv}\theta_v (B+f_{vw}v)=A+f_{wv}B+f_{wv}f_{vw}\theta_v (v)=\\
A+f_{wv}B+f_{wv}f_{vw}(A+f_{wv}v)=A+f_{wv}B+f_{wv}f_{vw}A+f_{wv}^2f_{vw}v
\text{, while}\end{multline*}
\begin{multline*}\theta_w\circ\theta_v\circ\theta_w(v)=\theta_w\circ\theta_v(v)=\\
\theta_w(A+f_{wv}w)=A+f_{wv}\theta_w(w)=A+f_{wv}(B+f_{vw}v)=A+f_{wv}B+f_{wv}f_{vw}v.\end{multline*}
When we combine them we see that 
\[A+f_{wv}B+f_{wv}f_{vw}A+f_{wv}^2f_{vw}v=A+f_{wv}B+f_{wv}f_{vw}v\iff \]
\[f_{wv}f_{vw}A+f_{wv}^2f_{vw}v=f_{wv}f_{vw}v. \]
By definition of $A$ it can not contain $v$ as a summand, which gives $f_{wv}f_{vw}A=0A$. By assumption, $A$ contains $u$ as a summand with non-zero coefficient, so $f_{wv}f_{vw}=0$. Since $\mathcal{R}$ is an integral domain, this implies $f_{wv}=0$ or $f_{vw}=0$. But $f$ only takes nonzero values by definition. 
If we allow $f$ to be zero on an edge $v\to w$ then straightforward calculations show that $R_f[wv]=R_f[vwv]$, showing that $f$ is not effective.
\label{double edges rep}
\end{rem}
	
We say that a word $\w$ is \emph{multiplicity free with respect to a vertex $v$} if $v$ appears 
at most once in $\w$. A word $\w$ is called \emph{multiplicity free} if it is multiplicity free with respect
to every vertex. An element $[\w]\in HK_{\G}$ is called \emph{multiplicity free} if $[\w]$ contains a
multiplicity free word. For $A\subset V(\G)$ define $\mathcal{MF}_A\subset \big(V(\G)\big)^*$ as the set 
of words which are multiplicity free with respect to all vertices in $A$. 
Let $\siso=\siso_{\G}\subset V(\G)$ denote the set of all sources and sinks. Multiplicity free words are easier to handle because they allow us to speak about \emph{the} position of a single vertex (assuming it exists).
Recall that a {\em subword} of a word
$v_1v_2\dots v_k$ is a word of the form $v_{i_1}v_{i_2}\dots v_{i_j}$ where
$1\leq i_1<i_2<\dots<i_j\leq k$. 

\begin{lem}\label{lemma:siso}
Let $A\subset \siso$ and $\w\in \big(V(\G)\big)^*$. Then $[\w]\cap\mathcal{MF}_A$ contains a subword of $\w$.
Furthermore, if $\w\sim \w'$ are both in $\mathcal{MF}_A$, then there exist a series of words 
$\w_i\in[\w]\cap \mathcal{MF}_A$, such that 
\[\w=\w_1\approx \w_2\approx\cdots\approx\w_k=\w'.\]
\end{lem}
	
\begin{proof}
If $\w$ is already in $\mathcal{MF}_A$, we take $\w'=\w$ and we are done. Assume that $\w=\w_1a\w_2a\w_3$ is a word with at least two occurrences of $a\in A$, and that $a$ is a source (if $a$ is a sink, a similar argument works with all words reversed). Because of our restrictions on $\G$,  for any $x\in V(\G)$ we have one of the following: 
\begin{enumerate}
	\item $x=a$ and $ax=aa=aaa=axa$;
	\item there is a directed edge from $a$ to $x$ and $ax=axa$;
	\item $a$ and $x$ are nonadjacent and $ax=aax=axa$.
\end{enumerate}
Note that we have the relation $ax=axa$ in all cases. Therefore, for $\w_2=x_1x_2\cdots x_k$ we have 
\[\w=\w_1ax_1x_2\cdots x_ka\w_3\sim \w_1ax_1ax_2\cdots x_ka\w_3\sim\cdots \sim \w_1ax_1ax_2a\cdots ax_ka\w_3\sim\]
\[\w_1ax_1ax_2a\cdots ax_k\w_3\sim\cdots\sim\w_1ax_1x_2\cdots x_k\w_3=\w_1a\w_2\w_3.\]
Thus we have lowered the multiplicity of $a$ in $\w$ while leaving all other vertices untouched. Let $\psi_a:\big(V(\G)\big)^*\to\mathcal{MF}_{\set{a}}\subset\big(V(\G)\big)^*$ be the function which removes superfluous $a$ as above, and let $\psi=\psi_A$ be the composition of all $\psi_a$ for $a\in A$. It is well-defined by the previous remark, and the fact that $A$ is a finite set. Then we obtain the desired subword $\psi(\w)$ of $\w$
contained in $[\w]\cap\mathcal{MF}_A$. This proves the first claim of the lemma.

Given two words $\w\sim\mathbf{\tilde{w}}$ there is, by definition, a sequence of words $\w_i\in[\w]$, $i\in\underline{k}$, such that 
\[\w=\w_1\approx \w_2\approx\cdots\approx\w_k=\mathbf{\tilde{w}}.\]
We want to show that if $\mathbf{a}\approx\mathbf{b}$, then $\psi(\mathbf{a})\approx\psi(\mathbf{b})$, but it suffices to show that $\psi_a(\mathbf{a})\approx\psi_a(\mathbf{b})$ for all $a\in A$. If $\mathbf{a}=\mathbf{b}$ this is trivial. Assume that $\mathbf{a}\neq\mathbf{b}$By definition, $\mathbf{a}\approx\mathbf{b}$ means that 
$\mathbf{a}=\mathbf{a}_l\mathbf{c}\mathbf{a}_r$ and $\mathbf{b}=\mathbf{a}_l\mathbf{d}\mathbf{a}_r$, where $\mathbf{c}=\mathbf{d}$ is an edge relation (or equality). Depending on the letters which appear in the edge relation, we consider different cases.
		
{\bf Case~1.} Relations $x^2=x, xyx=yxy=xy$ or $xy=yx$ for $x,y\not\in A$. In this case the application of 
$\psi_a$ affects neither $\mathbf{c}$ nor $\mathbf{d}$. Assume that $a\in A$ is a source which appears in exactly one
of the words $\mathbf{a}_l$ or $\mathbf{a}_r$. Then the application of $\psi_a$ to both $\mathbf{a}$ and 
$\mathbf{b}$ deletes all but the leftmost occurrences of $a$. If $a\in A$ is a source which appears in both
$\mathbf{a}_l$ and $\mathbf{a}_r$, then the application of $\psi_a$ to both $\mathbf{a}$ and 
$\mathbf{b}$ deletes all occurrences of $a$ in $\mathbf{a}_r$ and all but the leftmost occurrences of $a$ in 
$\mathbf{a}_l$. Similarly one considers the case when $a\in A$ is a sink. It follows that 
$\psi_a(\mathbf{a})\approx\psi_a(\mathbf{b})$.
			
{\bf Case~2.} Relations $a^2=a,ax=xa$ and $axa=xax=ax$ where $a\in A$ and $x\not\in A$. 
We assume that $a$ is a source (the case when $a$ is a sink is done similarly).
If $a$ appears in $\mathbf{a}_l$, then $\psi_a$ deletes all $a$ in $\mathbf{a}_r$, $\mathbf{c}$ and $\mathbf{d}$
(and leaves just the leftmost occurrences of $a$ in $\mathbf{a}_l$). Note that our edge relations
become $\varepsilon=\varepsilon$, $x=x$ and $x=x^2$, respectively. It follows that in this case $\psi_a(\mathbf{a})\approx\psi_a(\mathbf{b})$. If $a$ does not appear in $\mathbf{a}_l$, then $\psi_a$ deletes all 
$a$ in $\mathbf{a}_r$ and leaves the leftmost occurrences of $a$ in $\mathbf{c}$ and $\mathbf{d}$. Note 
that our edge relations become $a=a$, $ax=xa$ and $ax=xax$, respectively. It follows that in this case
we have $\psi_a(\mathbf{a})\approx\psi_a(\mathbf{b})$, which may be equality, depending on the case.

{\bf Case~3.} Relation $ab=ba$ where $a,b\in A$. Similarly to the above, the application of 
$\psi_a$ does the same thing to the subword $\mathbf{a}_l$ of both $\mathbf{a}$ and $\mathbf{b}$,
it does the same thing to  the subword $\mathbf{a}_r$ of both $\mathbf{a}$ and $\mathbf{b}$,
and it maps $ab=ba$ to either $ab=ba$ or $a=a$ or $b=b$ or $\varepsilon=\varepsilon$, depending on
whether $a$ or $b$ appear in $\mathbf{a}_l$ (if they are sources) or in $\mathbf{a}_r$ (if they are sinks).
In all cases we get that $\psi_a(\mathbf{a})\approx\psi_a(\mathbf{b})$, which may be equality, depending on the case.

{\bf Case~4.} Relation $aba=bab=ab$, where $a,b\in A$. Here $a$ is a source and $b$ is a sink.
Similarly to the above, the application of $\psi_a$ does the same thing to the subword $\mathbf{a}_l$ of both 
$\mathbf{a}$ and $\mathbf{b}$ and it does the same thing to  the subword $\mathbf{a}_r$ of both $\mathbf{a}$ 
and $\mathbf{b}$. Depending on the appearance of $a$ in $\mathbf{a}_l$ and $b$ in $\mathbf{a}_r$,
the relation is either mapped to $ab=ab$ or to $a=a$ or to $b=b$ or to $\varepsilon=\varepsilon$.
Again, in all cases we get that $\psi_a(\mathbf{a})\approx\psi_a(\mathbf{b})$. The claim follows.

\end{proof}
	
\begin{lem}
Let $\G$ be a simple directed graph with at most one edge between any pair of vertices.
Let $\G'\subset \G$ be a full subgraph and let $\w$ be a word such that $\mathfrak{c}[\w]=V(\G')$. Then 
exactly one of the following is true:
\begin{enumerate}
\item $\G'$ contains an oriented cycle and $[\w]^k$ are pairwise distinct for all $k\in \mathbb{N}$. 
\item $\G'$ contains no oriented cycles and $[\w]^{|V(\G')|}$ is the zero element in $HK_{\G'}$.
\end{enumerate}
\end{lem}

\begin{proof}
Assume that $\G'$ contains an oriented cycle $C$ (which then necessarily has length at least $3$).
Using $p:HK_\G\to HK_C$ it is enough to prove the first claim under the assumption
$\G=\G'=C$ (since if $p([\w]^k)$ are pairwise distinct for all $k$, then $[\w]^k$ are pairwise distinct for all $k$ as well). 

Let the vertices in $C$ be enumerated by $\underbar{n}$, such that there is a directed edge from $v_i$ to 
$v_{i+1}$ for all $i\in\underbar{n}$. To separate elements, we choose the representation $R_2$ (with the 
ground ring $\mathcal{R}=\mathbb{Z}$) and prove that the images $R_2([\w^k])$ are pairwise different. 
Recall that each $\theta_i$ maps $v_i$ to $2v_{i-1}$ and $v_j$ to $v_j$ for $j\neq i$.
It follows that there is a transformation $t:V(C)\to V(C)$ and a set of non-negative integers $m_1,\cdots,m_n$ such that  $R_2([\w])(v_i)=2^{m_i}v_{t(i)}$. Moreover, $\mathfrak{c}([\w])=V(C)$ implies that $m_i\geq 1$ for each $i$. We define the sequence $n_i$ by $R_2([\w])(v_{t^{i-1}(1)})=2^{n_i}v_{t^i(1)}$, and $\overline{n_i}=\sum_{j=1}^i n_i$, that is, $n_1=m_1$, $n_2=m_{t(1)}$, $n_3=m_{t^2(1)}$, etc. Then 
\[R_2([\w^k])(v_1)=2^{n_1}R_2([\w^{k-1}])(v_{t(1)})=\cdots=2^{\overline{n_k}}v_{t^k(1)}.\]  
Since the exponent is strictly increasing in $k$, and since $R_2$ is effective, it follows that the action of $R_2[\w^k]$ is different for different $k$, and hence $[\w^k]$ are pairwise distinct elements.

Now assume that $\G$ contains no oriented cycles. Then $HK_\G$ is a quotient of 
Kiselman's semigroup and the second claim follows from \cite[Lemma~12]{KM09}.
\end{proof}
	
\begin{thm}
There is a bijection between idempotents in $HK_\G$ and  full subgraphs of $\G$ which do not contain
any oriented cycles.
\end{thm}

\begin{proof}
Let $[\w]\in HK_\G$ be an idempotent and $\G'\subset\G$ be the full subgraph whose set of vertices is $\mathfrak{c}[\w]$.
Then $\G'$ contains no oriented cycles by the first claim of the previous lemma.

Conversely, let $\G'$ be a full subgraph of $\G$ which does not contain any oriented cycles.
Then $HK_{\G'}$ is a quotient of Kiselman's semigroup. In particular, $HK_{\G'}$ is finite and contains
a unique idempotent of maximal content, namely the zero element, see \cite{KM09}. The claim follows.
\end{proof}

\section{Main results}\label{s4}

A graph of type $A_n$ consists of linearly ordered pieces which glue together in sources and sinks in the interior. The number of linearly ordered pieces is one less than the number of sources and sinks. Since we know a lot about effective representations for the pieces (see the previous section), we would like to know what the gluing does to representations. 
	
\begin{defn}
Given a graph $\G$ and a vertex $a\in\G$ the \emph{source graph} $S_a$ is the full subgraph of $\G$ with the vertex set
\begin{displaymath}
V(S_a)=\{v\in\G\mid \text{ there exists a directed path from } v \text{ to } a\}. 
\end{displaymath}
\end{defn}

\begin{exmp}
Let $\G$ be the graph
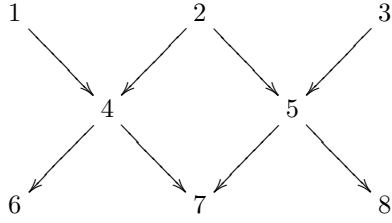
\[\xymatrix{1\ar[dr]&&\ar[dl]2\ar[dr]&&\ar[dl]3\\
&4\ar[dr]\ar[dl]&&\ar[dl]5\ar[dr]&\\
6&&7&&8}\]
Then $S_7$ is the graph 
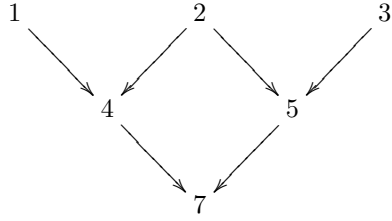
\[\xymatrix{1\ar[dr]&&\ar[dl]2\ar[dr]&&\ar[dl]3\\
&4\ar[dr]&&\ar[dl]5&\\
&&7&&}\]
and $S_6$ is the graph
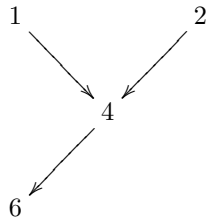
\[\xymatrix{1\ar[dr]&&\ar[dl]2\\
&4\ar[dl]\\
6}\]
\end{exmp}

\begin{lem}
For a vertex $a\in\G$ let $p_a:HK_\G\to HK_{S_a}$ be the canonical projection. Then for any $[\w]\in HK_\G$ we have
the equality $R_f([\w])(a)=R_f(p_a[\w])(a)$. 
\end{lem}

\begin{proof}
Let $x\in V(\G)$ and $v\in S_a$. From the definition of $R_f$ we have that the linear span $L$ of all $v\in S_a$ is invariant with respect to the action of
$HK_{\G}$. Furthermore, for any $x\in V(\G)\setminus S_a$ we have $R_f([x])(v)=v$ for all $v\in S_a$, which
means that $R_f([x])$ acts as the identity on $L$. It follows that the actions of  
$R_f([\w])$ and $R_f(p_a[\w])$ on $L$ coincide.
\end{proof}
		
If $\G=\cup_{i\in I}\G_i$ is a union of full subgraphs which pairwise do not have any common edges and
$f^i$ is a collection of weight functions, then there is a unique weight function $f$ on $\G$ whose restriction
to $\G_i$ coincides with $f^i$ for every $i$. We call $f$ the \emph{extension} of the $\{f^i:i\in I\}$.

A subgraph $\G'\subset\G$ is called {\em path complete} if every oriented path in $\G$ which
starts from a vertex of $\G'$ and ends at a vertex of $\G'$ is contained entierly in $\G$.
For example, if $\G=\xymatrix{1\ar[rr]&&2\ar[rr]&&3}$, then the full subgraph with vertices
$1$ and $2$ is path complete while the full subgraph with vertices $1$ and $3$ is not path complete. 

A path complete subgraph is always a full subgraph, but the converse is false, in general.

Let $\G'$ be a path complete subgraph of $\G$ and $f$ a weight function on $E(\G)$. Let $f'$ denote the restriction of
$f$ to $\G'$ (with the corresponding representation $R_{f'}$ of $HK_{\G'}$). 
Consider the representation $R_f$ of $HK_\G$ on $W=\bigoplus_{v\in\G}\mathcal{R}v$.
Let $\G''=\cup_{v\in\G'}S_v$ be the full subgraph of $\G$ whose set of vertices consists of all vertices of $\G$
from which there is an oriented (but maybe trivial) path to a vertex of $\G'$. Clearly,
$\G'$ is a subgraph of $\G''$. Finally, let $\G'''$ be the full subgraph of $\G$ whose set of vertices
coincides with the set of all vertices of $\G''$ which do not belong to $\G'$.
Then both $X=\bigoplus_{v\in\G''}\mathcal{R}v$ and $Y=\bigoplus_{v\in\G'''}\mathcal{R}v$ are 
invariant under the action of $HK_\G$. Indeed, by construction $\G''$ contains all arrows pointing to some vertex in $\G''$, which implies that $X$ is invariant under the action of $HK_\G$. The same reasoning is valid for $Y$ when we add the fact that $\G'$ is path complete. Let $\rho$ be the corresponding representation of $HK_\G$ on $X/Y$.

\begin{prop}\label{thm:path}
The representations $\rho$ and $R_{f'}\circ p$ of $HK_\G$ are isomorphic.
\end{prop}

\begin{proof}
Define the map $\Phi:X/Y\to \bigoplus_{v\in\G'}\mathcal{R}v$ as the unique $\mathcal{R}$-linear map which
sends $v+Y$ for $v\in V(\G')$ to $v$. This is obviously linear and bijective. The fact that it is
a homomorphism of $HK_\G$-modules follows directly from the definitions.
\end{proof}

This proposition says that, for every word $\w\in(V(\G))^*$, the minor of the matrix of $R_f([\w])$ 
corresponding to the basis vectors $\{v:v\in\G'\}$ coincides with the matrix
$R_{f'}(p[\w])$. This fails if $\G'$ is not path complete. For example, if we let 
$\G = \xymatrix{a\ar[r] & b\ar[r] & c}$ and $\G'$ be the full subgraph with vertices $a$ and $c$,
then for $\w =bc$ and $f\equiv 1$ we have  $R_1([bc])(c)=a$ while $R_{1'}(p[bc])(c)=0$.
	
\begin{prop}\label{thm:p-inj}
Assume that $\G$ is a union of two full subgraphs, $\G_1$ and $\G_2$, 
such that $V(\G_1)\cap V(\G_2)=\{a\}$ for some $a\in\siso_\G$. Then the map
$p=(p_1,p_2):HK_\G\to HK_{\G_1}\times HK_{\G_2}$, where 
$p_i:HK_\G\to HK_{\G_i}$, \mbox{$i=1,2$}, are projection morphisms, is injective.  
\end{prop}

\begin{proof}
Let $\G'$ denote the full subgraph of $\G$ with vertices $V(\G)\setminus\{a\}$ and define 
$\G'_1$ and $\G'_2$ similarly. Assume that $\w$ and $\w'$ are such that $p[\w]=p(\w')$. We need to show that this implies $[\w]=[\w']$, or, equivalently, $\w\sim\w'$.  Since $a\in\siso_\G$, by Lemma~\ref{lemma:siso} every element $[\w]\in HK_\G$ 
contains at least one word $\w'$ which is multiplicity free with respect to $a$, i.e. 
$\w'=\w_1\alpha\w_2$ where $\alpha\in\{\varepsilon, a\}$ and $\w_1,\w_2\in HK_{\G'}$. Each of $\w_1$ and $\w_2$ can be written on the form $\w_i=\x_i\y_i$ for some $\x\in HK_{\G_1'}$ and $\y\in HK_{\Gamma_2'}$.  We do the same with $\w'$ and denote the result with primes, that is $\w'=\w_1'\alpha\w_2'=\x_1'\y_1'\alpha\x_2'\y_2'$. First note that $a$ is in $\mathfrak{c}[\w]$ if, and only if, $a$ is in $\mathfrak{c}[\w']$.

We claim that $p^{-1}$ is given by $([\x_1\alpha\x_2,\y_1\alpha\y_2])\mapsto[\x_1\y_1\alpha\x_2\y_2]$. We need to prove two things; a) the image of $p$ consists of elements on that form, and b) the map $p^{-1}$ is well-defined. The entire codomain of $p$ consists of elements on the form $([\x],[\y])$, for some $\x\in HK_{\G_1}$ and $\y\in HK_{\G_2}$. We can factor $\x=\x_1\alpha_1\x_2$ and $\y=\y_1\alpha_2\y_2$ by similar arguments as for $\w$ above. When we restrict to the image, we impose the condition that $\alpha_1=\alpha_2=\alpha$. On the other hand, any element on the form $([\x_1\alpha\x_2,\y_1\alpha\y_2])$ is the image under $p$ of the element $\x_1\y_1\alpha\x_2\y_2$, so the image does indeed consist precisely of the elements $\set{([\x_1\alpha\x_2,\y_1\alpha\y_2])}$. To prove that $p^{-1}$ is well-defined, we divide into two cases, depending on if $\alpha=\varepsilon$ or $\alpha=a$. By assumption and construction we have the following relations, which we use without further comment:
\begin{align*}
\x_1\alpha\x_2&\sim\x_1'\alpha\x_2',\\
\y_1\alpha\y_2&\sim\y_1'\alpha\y_2'\text{ and}\\
\x\y&\sim\y\x\text{ for all decorations on $\x$ and $\y$}.
\end{align*}

If $\alpha=\varepsilon$, then $\x_1\y_1\x_2\y_2\sim\x_1\x_2\y_1\y_2$ and it suffices to show that $\x_1\x_2\y_1\y_2\sim\x_1'\x_2'\y_1'\y_2'$, but this follows in two steps from the relations above. 

Now assume that $\alpha=a$. Then
\begin{multline*}
\x_1'\y_1'\alpha\x_2'\y_2'\sim
\x_1'\y_1'\alpha\y_2'\x_2'\sim
\x_1'\y_1\alpha\y_2\x_2'\sim \\\sim
\y_1\x_1'\alpha\x_2'\y_2=\y_1\x_1'\alpha\x_2'\y_2\sim
\y_1\x_1\alpha\x_2\y_2\sim
\x_1\y_1\alpha\x_2\y_2. 
\end{multline*}
Hence $\w'\sim\w$ and the claim follows.
\end{proof}

\begin{thm}\label{thmmain}
Let $\G$, $\G_1$ and $\G_2$ be as in Proposition~\ref{thm:p-inj}. Let $f_1$ and $f_2$ be weight functions for
$\G_1$ and $\G_2$, respectively and $f$ be the extension of $\{f_1,f_2\}$ to $\G$. Then the representation
$R_f$ of $HK_{\G}$ is effective if, and only if, the representations $R_{f_1}$ of $HK_{\G_1}$ and $R_{f_2}$ 
of $HK_{\G_2}$ are effective.
\end{thm}
		
\begin{proof}
Let $\G'$, $\G'_1$ and $\G'_2$ be as in the proof of Proposition~\ref{thm:p-inj}. Recall that the construction is such that any edge of $\Gamma$ is either in $\Gamma_1$ or $\Gamma_2$.
We start with the ``only if'' part. Both $\G_1$ and $\G_2$ are path complete, which implies that
\[R_f[\w](v)=\begin{cases}
R_{f_i}(p_i[\w])(v), &\text{ if }v\in\G_i';\\
R_{f_1}(p_1[\w])(a)+R_{f_2}(p_2[\w])(a), &\text{ if }v=a\text{ and }a\in\mathfrak{c}[\w];\\
0 & \text{otherwise}
\end{cases}\]
Now assume that $R_{f_1}$ is not effective, i.e. there exists $[\w_1]\neq[\w_2]\in HK_{\G_2}$ such that $R_{f_1}([\w_1])=R_{f_1}([\w_2])$. Then $p_2[\w_1]=p_2[\w_2](=[a]\text{ or }[\varepsilon])$, and
\[R_f([\w_1])(v)=\begin{cases}
R_{f_1}(p_1[\w_1])(v), &\text{ if }v\in\G_1';\\
R_{f_2}(p_2[\w_1])(v), &\text{ if }v\in\G_2';\\
R_{f_1}(p_1[\w_1])(a)+R_{f_2}(p_2[\w_1])(a), &\text{ if }v=a\text{ and }a\in\mathfrak{c}[\w].\\
0 & \text{otherwise}
\end{cases}\]

The ``if'' part follows by combining Propositions~\ref{thm:path} and \ref{thm:p-inj}
(note again that both $\G_1$ and $\G_2$ are path connected in $\G$). 
Assume that the representations $R_{f_1}$ of $HK_{\G_1}$ and $R_{f_2}$ of $HK_{\G_2}$ are effective
and consider the representation $R_f$ of $HK_{\G}$. Let $[\w]\neq[\w']$ be two elements of
$HK_{\G}$. Assume $R_f([\w])=R_f([\w'])$. Then Propositions~\ref{thm:path} and the arguments from the
first part of the proof imply that $R_{f_1}(p_1([\w]))=R_{f_1}(p_1([\w']))$ and
$R_{f_2}(p_2([\w]))=R_{f_2}(p_2([\w']))$. Since both $R_{f_1}$ and $R_{f_2}$ are effective,
we get $p_1([\w])=p_1([\w'])$ and $p_2([\w])=p_2([\w'])$. Now from Proposition~\ref{thm:p-inj}
we get $[\w]=[\w']$. The claim follows.
\end{proof}

This statement can now be iterated as follows. Assume that $\G$ is a union of full
subgraphs, $\displaystyle \G=\bigcup_{i=1}^n \G_i$, where $n>1$, such that each pair of different subgraphs
does not have any common edges and, moreover, we assume that for every $k=2,3,\dots,n$ there
is $a_{k}\in\siso_\G$ such that 
\[V(\G_k)\cap\big(\bigcup_{i=1}^{k-1}V(\G_i)\big)=\{a_{k}\}.\]
In this case we say that $\G$ satisfies the \emph{gluing condition}.

For example, let $\G$ be of type $A_n$. If we define $\G_i$ to be the maximal connected linearly ordered
full subgraphs of $\G$, then $\G$ satisfies the gluing condition with respect to these subgraphs as
illustrated below (here $\Gamma_i$ is a subgraph with vertices between $a_i$ and $a_{i+1}$,
where $a_0$ is the leftmost vertex and $a_{k+1}$ is the rightmost vertex):

\begin{equation}\label{eq21}{\tiny
\xymatrix{a_0\ar[r]&\cdots\ar[r]&a_1&\cdots\ar[l]&a_3
\ar[l]\ar[r]&\cdots \cdots\ar[r]& a_k &\cdots\ar[l] & a_{k+1} \ar[l]\\
&\G_1 & & \G_2 & & \cdots & &\G_{k+1}} 
}\end{equation}

This implies the following corollary which answers \cite[Question~8]{GM11}.

\begin{cor}
Let $\G$ be of type $A_n$. Then the representation $R_1$ of $HK_{\G}$ is effective.
\end{cor}

\begin{proof}
This follows from Theorem~\ref{thmmain} and  \cite[Subsection~3.2]{GM11} by induction on the number of internal sinks and sources.
\end{proof}

As a byproduct of the proof of Proposition~\ref{thm:p-inj} we get formulae for certain cardinalities.
An element $[\w]$ is said to have \emph{maximal content} if $\mathfrak{c}[\w]=V(\G)$. 
The subsemigroup consisting of elements with maximal content is denoted by $\m(\G)$. 

\begin{thm}\label{thmproduct}
Assume that $\displaystyle \G=\bigcup_{i=1}^k \G_i$ satisfies the gluing condition. 
Then \[|\m(\G)|=\prod_{i=1}^k |\m(\G_i)|.\] 
\end{thm}
	
\begin{proof}
By induction, it suffices to prove the claim for $\G=\G_1\cup\G_2$. All we need to show is that the restricted 
map $p:\m(\G)\to \m(\G_1)\times\m(\G_2)$ is a bijection. Since we know that it is injective
(by Proposition~\ref{thm:p-inj}), we only have to establish its surjectivity. 
Let $([\x],[\y])\in \m(\G_1)\times\m(\G_2)$. By definition, $a\in\mathfrak{c}[\x]$ and
$a\in\mathfrak{c}[\y]$, so there are words $\x\sim\x_1a\x_2,\y\sim\y_1a\y_2$ which are multiplicity free 
with respect to $a$. Then  $[\x_1\y_1a\x_2\y_2]$ is the  preimage of $([\x],[\y])$. The claim follows.
\end{proof}
	
The Catalan numbers $C_n=\frac{1}{n+1}{\binom{2n}{n}}$ are the cardinalities of the HK-monoids of linearly 
ordered graphs, see \cite[Theorem~1(vi)]{GM11}. They can also be used to calculate the cardinality of any 
HK-monoid of type $A_n$. 
	
Let $\displaystyle\G=\bigcup_{i=1}^{k+1}\G_i$ be a graph of type $A_n$ as in \eqref{eq21} or its opposite. Specifically, it has $k+2$ sinks and sources, out of which $k$ are gluing linearly ordered pieces together. Let $[\w]$ be an element in $HK_\G$. Define its \emph{signature} as $\mathfrak{s}[\w]:=\mathfrak{c}[\w]\cap\mathcal{S}_\Gamma$. Similarly we define the local signature as $\mathfrak{s}_i[\w]:=\mathfrak{c}[\w]\cap\mathcal{S}_{\Gamma_i}$ for all $1\leq i\leq k+1$. Clearly, the equality $\mathfrak{s}_i[\w]=\mathfrak{s}[\w]\cap\set{i-1,i}$ holds for all $1\leq i\leq k+1$. In the calculations ahead, we use the set $Q$ to keep track of signatures and the functions $\delta_i$ to calculate local signatures from a signature. The set $X_i(Q)$ will be the set of all elements in $HK_{\G_i}$ with local signature determined by $Q$. The functions $c_i$ will count the cardinalities of the sets $X_i$. 
 
For a subset $Q\subset\{1,2,3,\cdots,k\}$ and $i\in\{1,2,3,\cdots,k+1\}$ set
\[\delta_i(Q)=\begin{cases}
		1, & \text{ if }i\in Q;\\
		0, & \text{otherwise};
\end{cases}\]
Let $l_i$  be the length of the linearly ordered piece between sources and sinks labeled $i$ and $i+1$, in terms of number of vertices. Alternatively $l_i$ is one more than the number of arrows in that piece. Let the functions $c_i$ be defined as follows:
\[c_i(Q)=\begin{cases}
		C_{l_i-1}, & \delta_{i-1}(Q)=\delta_i(Q)=0; \\
		C_{l_i}-C_{l_i-1}, &  \delta_{i-1}(Q)+\delta_i(Q)=1; \\
		C_{l_i+1}-2C_{l_i}+C_{l_i-1}, & \delta_{i-1}(Q)=\delta_i(Q)=1, 
\end{cases}\]
for $2\leq i\leq k$, and
\[c_1(Q)=\begin{cases}
	C_{l_1}, & \delta_1(Q)=0;\\
	C_{l_1+1}-C_{l_1}, & \delta_1(Q)=1;
\end{cases}\quad
	c_{k+1}(Q)=\begin{cases}
	C_{l_{k+1}}, & \delta_{k}(Q)=0;\\
	C_{l_{k+1}+1}-C_{l_{k+1}}, & \delta_{k}(Q)=1.
\end{cases}\]
	
\begin{cor}
Let $\G$ be as above. Then we have
\begin{enumerate}[(i)]
\item $\displaystyle |\m(\G)|=\prod_{i=1}^{k+1}C_{l_i}$,
\item $\displaystyle |HK_\G|=\sum_{Q\subset \underline{k+1}}\prod_{i=1}^{k+1}c_i(Q)$,
\item if $\G=\mathcal{S}_{\G}$, that is,
\[\G = 1\rightarrow 2\leftarrow 3\rightarrow 4\leftarrow \cdots n \quad\text{ or }
\quad\G=1\leftarrow 2\rightarrow 3\leftarrow 4 \rightarrow \cdots n,\]
then $|HK_\G|=F_{2n+1}$ is the $(2n+1)$-th Fibonacci number (where $F_1=F_2=1)$.
\end{enumerate}
\end{cor}
	
\begin{proof}
To prove the first claim we need to know $\m(\G_i)$ and then apply Theorem~\ref{thmproduct}. 
Since $\G_i$ is linearly ordered, the elements of $HK_{\G_i}$ are in bijection with order 
preserving and order decreasing transformations on a set with $l_i+1$ elements, see \cite[Theorem~1(vii)]{GM11}. 
This bijection restricts to a bijection between $\m(\G_i)$ and transformations $\tau$ such that $\tau(j)<j$ 
for all $j\neq 1$. Changing $j$ in the domain to $j-1$ gives a bijection between $\m(\G_i)$ with order 
preserving and order decreasing transformations on a set with $l_i$ elements and hence $|\m(\G_i)|=C_{l_i}$
by \cite[Theorem~1(vi) and (vii)]{GM11}.

For the second claim we have to work more. Assume that $\G=\G_1\cup\G_2$ satisfies the gluing condition
with $a\in\siso_\G$ as the common vertex. Then $HK_\G$ splits into two subsets (in fact subsemigroups) \mbox{$A_\G:=\{[\w]\in HK_\G\mid a\in\mathfrak{c}[\w]\}$} and $B_\G:=\{[\w]\in HK_\G\mid a\not\in\mathfrak{c}[\w]\}$. Similarly, $HK_{\G_i}=A_{\G_i}\cup B_{\G_i}$ for $i=1,2$. The function $p$ in Proposition~\ref{thm:p-inj} restricts to bijections 
\[p:A_\G\to A_{\G_1}\times A_{\G_2}\quad  \text{ and }\quad p:B_\G\to B_{\G_1}\times B_{\G_2}.\]
By the multiplicative principle we have $|A_\G|=|A_{\G_1}|\cdot |A_{\G_2}|$ and 
\mbox{$|B_\G|=|B_{\G_1}|\cdot|B_{\G_2}|$}. 

For $Q\subset \{1,2,\cdots,k\}$ and $i=2,\dots,k+1$, define $a_Q:=\set{a_i\mid i\in Q}$ and 
\begin{displaymath}
X_i(Q):=\set{[\w]\in HK_{\G_i}\mid \mathfrak{s}_i[\w]=a_Q\cap\set{a_{i-1},a_i}} 
\end{displaymath}
and set
\begin{gather*}
X_1(Q):=\{[\w]\in HK_{\G_1}\mid \mathfrak{s}[\w]=a_Q\cap\{a_1\}\} \\
X_{k+1}(Q):=\{[\w]\in HK_{\G_{k+1}}\mid \mathfrak{c}[\w]=a_Q\cap\{a_k\}\}.
\end{gather*}
Since every element has exactly one signature, we can sum over all possible signatures. For a fixed signature the multiplicativity which holds for two pieces translates to multiplicativity over all pieces. Put together we get the following formula:
\begin{equation}\label{eq323}
|HK_\G|=\sum_{Q\subset\underline{k}\setminus\{1\}}\prod_{i=1}^k|X_i(Q)| 
\end{equation}
To compute cardinalities of $X_i(Q)$ (and show that $c_i(Q)=|X_i(Q)|$, as claimed) we have to consider several cases. We start with the case
$i\neq1,k+1$.
		
\begin{enumerate}
\item Assume that  $i\not\in Q$ and $i+1\not\in Q$. If $a_i,a_{i+1}\not\in\mathfrak{c}[\w_i]$, then $[\w_i]$ can be thought of as living in the smaller HK-monoid $HK_{\G_i\setminus\{a_i,a_{i+1}\}}$, and vice versa. Since $\G_i$ was assumed to be linearly ordered, it will be the case for $\G_i\setminus\{a_i,a_{i+1}\}$ as well. Note that all $\G_i$ have length $l_i\geq 2$ (if $\G$ has at least two vertices). Therefore $|HK_{\G_i\setminus\{a_i,a_{i+1}\}}|=C_{l_i-1}$.
\item Assume that exactly one of $i,i+1$ is in $Q$. Without loss of generality assume $i\in A$. Similarly as above we count the number of elements that \emph{do not} contain $a_{i+1}$. That cardinality is $C_{l_i}$. However, we need to exclude the elements that do not contain $a_i$, leaving us with exactly $C_{l_i}-C_{l_i-1}$.
\item If both of $i,i+1$ are in $Q$, we use the inclusion exclusion formula to get $C_{l_i+1}-2(C_{l_i}-C_{l_i-1})-C_{l_i-1}=C_{l_i+1}-C_{l_i}+C_{l_i-1}$ elements.
\end{enumerate}
		
When $i=1$ or $i=k+1$, we get the following two cases. Depending on $i$, let $a=a_1$ or $a_k$.
\begin{enumerate}
\item Elements in $HK_{\G_i}$ that do not contain $a$. There are $C_{l_i}$ such elements.
\item Elements in $HK_{\G_i}$ that do contain $a$. By exclusion there are $C_{l_i+1}-C_{l_i}$ such elements.
\end{enumerate}
Now the second claim of our theorem follows from \eqref{eq323} and the definition of $c_i(A)$.
		
We will prove the last claim by showing that it satisfies the same recursion formula as the odd Fibonacci numbers and has the same initial values. Note that if the number of vertices $n$ is fixed there are only two possibilities for a graph of type $A_n$ to have alternating sinks or sources (the first vertex can either be a sink or a source). However, these graphs are opposite to each other, so the cardinalities of the corresponding HK-monoids have to be the same 
by \cite[Theorem~1(v)]{GM11}. Let $\mathcal{A}_n$ be the graph of type $A_n$ with alternating sinks and sources
whose vertices are $v_i$, $i\in\underline{n}$, and we assume that $v_1$ is a source. Let $f_n=|HK_{\mathcal{A}_n}|$. The first two $\mathcal{A}_n$ are $\mathcal{A}_0=$ the empty graph, and $\mathcal{A}_1=v_1$. This gives $f_0=|\{\varepsilon\}|=1=F_{2\cdot0+1}$ and $f_1=|\{\varepsilon,[v_1]\}|=2=F_{2\cdot1+1}$. The Fibonacci numbers satisfy the recursion formula
\[F_{n+2}=F_{n+1}+F_n=2F_n+F_{n-1}=3F_n-F_{n-2}.\]
Thus we want to show that $f_{n+1}=3f_n-f_{n-1}$. We can separate the elements of 
$HK_{\mathcal{A}_{n+1}}$ into two groups. Either the content of an element contains $v_n$ or it does not. 
\begin{enumerate}
\item Assume $v_n\not\in\mathfrak{c}[\w]$, then $[\w]$ equals an element from $HK_{\mathcal{A}_{n-1}}$ 
multiplied with either $\varepsilon$ or $v_{n+1}$. Thus there are  $f_{n-1}\cdot2=2f_{n-1}$ such elements.
\item Assume $v_n\in\mathfrak{c}[\w]$, then $[\w]$ is a product of an element
from $HK_{\mathcal{A}_n}$ \emph{containing} $v_n$ with an element from $HK_{\mathcal{B}}$ \emph{containing} $v_n$,
where $\mathcal{B}$ is the full subgraph of $\mathcal{A}_{n+1}$ with vertices $\{v_n,v_{n+1}\}$. 
We have $(f_n-f_{n-1})\cdot3=3f_n-3f_{n-1}$ such elements. The '3' in the formula corresponds to how $v_{n+1}$ relates to $v_n$. Either there is no $v_{n+1}$ or there is exactly one, and in that case it only matters if it comes before or after $v_n$.
\end{enumerate} 
This implies $f_{n+1}=3f_n-f_{n-1}$ and completes the proof of our theorem.
\end{proof}
	
The sequence $F_{2n+1}$ was guessed with the help of \cite{OEIS} and has been found independently by Grensing \cite{Gre12}.
	
\begin{cor}
Let $\G$ be of type $A_n$. Then the number of multiplicity free elements in $HK_\G$ is the Fibonacci number $F_{2n+1}$. In particular, it does not depend on the orientation of edges in $\G$.
\end{cor}

\begin{proof}
Lemma \ref{lemma:siso} tells us that every element $[\w]$ contains a word $\w'$ which is multiplicity free \emph{with respect to every source and sink}. Since $\mathcal{A}_n$ has only sources and sinks, every element of $HK_{\mathcal{A}_n}$ is in fact multiplicity free. Since $|HK_{\mathcal{A}_n}|=F_{2n+1}$ by the previous theorem, we need to show that there is a bijection between multiplicity free elements of $HK_\G$ and $HK_{\mathcal{A}_n}$. Let $\w$ be a multiplicity free word and assume that $\G$ is enumerated canonically, i.e. such arrows connect vertices of difference 1. Then for each $i\in\underline{n-1}$ exactly one of the following holds.
\begin{enumerate}[(1)]
\item $v_{i+1}\not\in\mathfrak{c}[\w]$.
\item $v_i\not\in\mathfrak{c}[\w]$ but $v_{i+1}\in\mathfrak{c}[\w]$.
\item Both $v_i,v_{i+1}\in\mathfrak{c}[\w]$ and $v_i$ appears before $v_{i+1}$ in $\w$.
\item Both $v_i,v_{i+1}\in\mathfrak{c}[\w]$ and $v_{i+1}$ appears before $v_{i}$ in $\w$.
\end{enumerate}
We claim that these properties do not depend on the choice of a \emph{multiplicity free} word $\w'\in[\w]$. 
For the first two properties the claim is obvious. Let $\G_i$ be the complete subgraph of $\G$ whose vertices are $v_i,v_{i+1}$, and consider the map $p:\big(V(\G)\big)^*\to\big(V(\G_i)\big)^*$ defined by deletion of vertices not in $\G_i$. It is clear that
\begin{enumerate}[(a)]
\item If $\w$ is multiplicity free, then so is $p(\w)$.
\item If $\w\sim\w'$ then $p(\w)\sim p(\w')$.
\end{enumerate}
This means that under $p$ all multiplicity free words that contain both $v_i$ and $v_{i+1}$ are mapped to $v_iv_{i+1}$ or $v_{i+1}v_i$. However, there is an edge between $v_i$ and $v_{i+1}$, so $v_iv_{i+1}\not\sim v_{i+1}v_i$, proving 
our claim for the third and the fourth properties.
		
We will show that multiplicity free elements are uniquely determined by the relative positions of 
$v_i$ and $v_{i+1}$ for each $i$. Set
\[\w_1=\begin{cases}v_1,&\text{ if }v_1\in\mathfrak{c}[\w];\\ \varepsilon,&\text{ if }v_1\not\in\mathfrak{c}[\w];\end{cases}\]
\[\w_{i+1}=\begin{cases}%
\w_i,&\text{ if }v_{i+1}\not\in\mathfrak{c}[\w];\\
\w_iv_{i+1},&\text{ if }v_i\not\in\mathfrak{c}[\w]\text{ or }v_i\text{ appears to the left of  }v_{i+1}\text{ in }\w;\\
v_{i+1}\w_i,&\text{ if }v_{i+1}\text{ appears to the left of }v_i\text{ in }\w.
\end{cases}\]
Let $\mathfrak{M}\subset HK_\G$ denote the set of all multiplicity free elements and define the map $\phi:\mathfrak{M}\to HK_{\mathcal{A}_n}$  by $\phi([\w])=[\w_n]$, where $\w$ is multiplicity free and $\w_n$ is defined 
from $\w$ by the above above. Note that $\phi$ is well-defined because it only uses invariant properties. Since $\phi:HK_{\mathcal{A}_n}\to HK_{\mathcal{A}_n}$ is the identity and the multiplicity free \emph{words} are the same for all $\G$ (over the same vertices) it follows that $\phi$ is surjective for any domain. To show that $\phi$ is injective it suffices to show that $\w\sim\w_n$ (in $HK_\G$) for each multiplicity free $\w$. We show this by deforming $\w$ into $\w_n$. Clearly $\w$ can be factorized as $\w=\x_1\w_1\y_1$, where 
$v_1\not\in\mathfrak{c}([\x_1]),\mathfrak{c}([\y_1])$. Assume that $\w\sim\x_i\w_i\y_i$ where $\{v_1,\cdots,v_i\}\cap\big(\mathfrak{c}([\x_i])\cup\mathfrak{c}([\y_i])\big)=\emptyset$. 
Note that $\w_i$ commutes with every vertex  \emph{except} $v_{i+1}$.
\begin{enumerate}
\item If $v_{i+1}\not\in\mathfrak{c}([\w])$, then $\w_{i+1}=\w_i$.
\item If $v_i\not\in\mathfrak{c}([\w])$ but $v_{i+1}\in\mathfrak{c}([\w])$ then $\w_i$ commutes with every vertex of $\set{v_{i+1},\cdots,v_n}$ and we may move it so that it ends up just to the left of $v_{i+1}$. 
\item If both $v_i,v_{i+1}\in\mathfrak{c}([\w])$ and $v_i$ appears before $v_{i+1}$ in $\w$ then $v_{i+1}$ is in $y_i$ and $\w_i$ commutes with every vertex in $\y_i$ preceding $v_{i+1}$. Hence, using edge relations, we may move  $\w_i$ so that it ends up just to the left of $v_{i+1}$.
\item If both $v_i,v_{i+1}\in\mathfrak{c}([\w])$ and $v_{i+1}$ appears before $v_{i}$ in $\w$, then $v_{i+1}$ is in $x_i$, and we move $\w_i$ 
so that it ends up just to the right of $v_{i+1}$.
\end{enumerate}
In all cases we find that
\[\w\sim\x_i\w_i\y_i\sim\x_{i+1}\w_{i+1}\y_{i+1}\]
for some $\x_{i+1}$ and $\y_{i+1}$ such that $\{v_1,\cdots,v_i,v_{i+1}\}
\cap\big(\mathfrak{c}([\x_{i+1}])\cup\mathfrak{c}([\y_{i+1}])\big)=\emptyset$. Proceeding inductively, we get
$\w\sim\x_n\w_n\y_n$ such that $\mathfrak{c}([\x_n])=\mathfrak{c}([\y_n])=\emptyset$, so $\w\sim\w_n$.
\end{proof}
	
To illustrate the process that turns $\w$ into $\w_n$ let $\w=v_3v_6v_1v_4v_2$.

\begin{tabular}{lll}
$\x_1\w_1\y_1=(v_3v_6)(v_1)(v_4v_2)$&$\x_2\w_2\y_2=(v_3v_6v_4)(v_1v_2)()$&$\x_3\w_3\y_3=()(v_3v_1v_2)(v_6v_4)$\\
$\x_4\w_4\y_4=(f)(v_3v_1v_2v_4)()$&$\x_5\w_5\y_5=(v_6)(v_3v_1v_2v_4)()$&$\x_6\w_6\y_6=()(v_3v_1v_2v_4v_6)()$
\end{tabular}
	
As a consequence of the previous corollary we have $F_{2n+1}\leq|HK_\G|$ for any $\G$. We also have $|HK_\G|\leq C_{n+1}$ by \cite[Theorem~1(vi)]{GM11}. It seems plausible that more sources and sinks corresponds to a smaller monoid. The following theorem makes this idea more precise.
	
\begin{thm}
Let $\G=\G_1\cup\G_2$ be an edge disjoint union of graphs such that $\G_1\cap\G_2=\{a\}$ is a source or a sink. Let $\overleftarrow{\G_2}$ be the graph obtained from $\G_2$ by reversing the direction of all edges and let $\tilde{\G}=\G_1\cup\overleftarrow{\G_2}$. Then $|HK_\G|\leq|HK_{\tilde{\G}}|$ and the equality holds if and only 
if  $a$ is isolated in at least one of $\G_1,\G_2$.
\end{thm}	

\begin{proof}
For any HK-monoid $HK_{\G'}$ and $a\in V(\G')$ define 
\begin{align*}
HK_{\G'}^0&=\{[\w]\in HK_{\G'}| a\not\in\mathfrak{c}[\w]\},\\
HK_{\G'}^1&=\{[\w]\in HK_{\G'}| a\in\mathfrak{c}[\w]\text{ and }[\w]\text{ is multiplicity free with respect to }a\},\text{ and}\\
HK_{\G'}^2&=\{[\w]\in HK_{\G'}|[\w]\text{ is not multiplicity free with respect to }a\}.
\end{align*}
Clearly $HK_{\G'}$ is a disjoint union of $HK_{\G'}^0$, $HK_{\G'}^1$ and $HK_{\G'}^2$. 
		
Because $|HK_{\G'}|=|HK_{\overleftarrow{\G'}}|$ by \cite[Theorem~1(v)]{GM11}, it follows that
\[|HK_{\G}^0|=|HK_{\G_1\setminus\{a\}}||HK_{\G_2\setminus\{a\}}|=|HK_{\G_1\setminus\{a\}}||HK_{\overleftarrow{\G_2\setminus\{a\}}}|=|HK_{\tilde{\G}}^0|.\]
To see that $|HK_{\G}^1|=|HK_{\tilde{\G}}^1|$ observe that any element $[\w]$ in $HK_{\G}^1$ has a word of the form $\w=\x_1\y_1a\x_2\y_2$ for some $\x_1,\x_2\in HK_{\G_1\setminus\{a\}},\y_1,\y_2\in HK_{\G_2\setminus\{a\}}$. Similarly for $HK_{\tilde{\G}}^1$. This is true because there are no edges between $\G_1\setminus\{a\}$ and $\G_2\setminus\{a\}$. Let $\phi:HK_\G^1\to HK_{\tilde{\G}}^1$ be defined by $\phi[\x_1\y_1a\x_2\y_2]=\x_1\overleftarrow{\y_2}a\x_2\overleftarrow{\y_1}$, where $\overleftarrow{\y}$ is the reverse of $\y$. Note that $\phi$ is a bijection if it is well-defined. To see that it is well-defined, assume $\x_1\y_1a\x_2\y_2\sim\x_1'\y_1'a\x_2'\y_2'$. By taking the projection morphisms (onto $HK_{\G_1}$ and $HK_{\G_2}$, respectively) we obtain $\x_1a\x_2\sim\x_1'a\x_2'$ and $\y_1a\y_2\sim\y_1'a\y_2'$. Note that $\y\sim\y'\iff\overleftarrow{\y}\sim\overleftarrow{\y'}$. We have:		
\begin{multline}\phi(\w)=%
\phi(\x_1\y_1a\x_2\y_2)=%
\x_1\overleftarrow{\y_2}a\x_2\overleftarrow{\y_1}\sim%
\x_1\overleftarrow{\y_2}a\overleftarrow{\y_1}\x_2=%
\x_1\overleftarrow{\y_1a\y_2}\x_2\sim\\
\x_1\overleftarrow{\y_1'a\y_2'}\x_2=%
\x_1\overleftarrow{\y_2'}a\overleftarrow{\y_1'}\x_2\sim%
\overleftarrow{\y_2'}\x_1a\x_2\overleftarrow{\y_1'}\sim%
\overleftarrow{\y_2'}\x_1'a\x_2'\overleftarrow{\y_1'}\sim%
\x_1'\overleftarrow{\y_2'}a\x_2'\overleftarrow{\y_1'}=\phi(\w').
\end{multline}
This implies that $|HK_{\G}^1|=|HK_{\tilde{\G}}^1|$.

Thus we have the inequality
\[|HK_\G|=|HK_\G^1|+|HK_\G^2|=|HK_\G^1|=|HK_{\tilde{\G}}^1|\leq|HK_{\tilde{\G}}^1|+|HK_{\tilde{\G}}^2|=|HK_{\tilde{\G}}|.\]
It remains to show that the inequality is an equality precisely when $a$ is isolated in at least one of $\G_1$ and $\G_2$. Assume that $a$ is isolated in $\G_1$. Then $a$ is a source in $\tilde{\G}$ if it is a sink in $\G$ and a sink in $\tilde{\G}$ if it is a source in $\G$. In any case, the set $HK_{\tilde{\G}^2}$ is empty, and the inequality is an equality. The case when $a$ is isolated in $\G_2$ follows by symmetry. Note that $a$ can be isolated in $\G_1$ and $\G_2$ simoultaneously, which is the case precisely when $a$ is both a sink and a source at the same time.

Now assume that $a$ is \emph{not} isolated in $\G_1$ or $\G_2$. Then there exists one of the following subgraphs of $\G$, depending on if $a$ is a sink or a source:
\[\xymatrix{b\ar[r]&a&\ar[l]c&\text{or}&b&a\ar[l]\ar[r]&c},\]
for some $b$ in $\G_1$ and $c$ in $\G_2$. In $\tilde{G}$, the same subgraphs turn into
\[\xymatrix{b\ar[r]&a\ar[r]&c&\text{or}&b&a\ar[l]&c\ar[l]}.\]
In either case we have $[abca]\in HK_{\tilde{\G}}^2$, proving that the inequality is in fact strict.
\end{proof}
Note that if $\G$ is cycle free, then so is $\tilde{\G}$. Thus only finitely many elements can make up the difference between the corresponding Hecke-Kiselman monoids. 

We illustrate the use of the theorem by the following figure with graphs of type $A_5$. We put one graph above another if we based on the theorem, together with reversals or graph isomorphisms, can tell that the Hecke-Kiselman monoid of the first has more elements than the Hecke-Kiselman monoid of the second. 
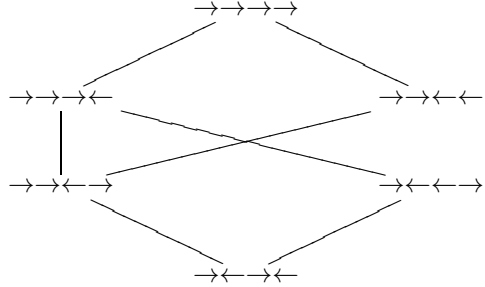
\[\xymatrix{&\ra\ra\ra\ra\ar@{-}[dl]\ar@{-}[dr]\\
\ra\ra\ra\la\ar@{-}[d]\ar@{-}[drr]&&\ra\ra\la\la\ar@{-}[dll]\\
\ra\ra\la\ra\ar@{-}[dr]&&\ra\la\la\ra\ar@{-}[dl]\\
&\ra\la\ra\la}\]

\section{Limits and extensions of the method}\label{s5}

The main idea of taking smaller graphs, whose HK-monoids have known effective representations, and gluing them together has a few limitations:
\begin{enumerate}
\item We only know effective representations for a few types of graphs \emph{that themselves are not reached in this way}. So far we have the linearly ordered graphs of type $A_n$ and the graphs $\kappa_n$ of the Kiselman monoids.
\item When we glue two subgraphs together we are limited to gluing sinks to sinks and sources to sources. In particular we can not handle trees in general, but only trees which branch in sources and sinks.
\item When we glue two subgraphs together the intersection must consist of one vertex. In particular we can not handle cycles (which give infinite monoids $HK_\G$, but nevertheless could have finite effective dimension).
\end{enumerate}
We may enlarge the set of ``building blocks'' by considering the following family of graphs:
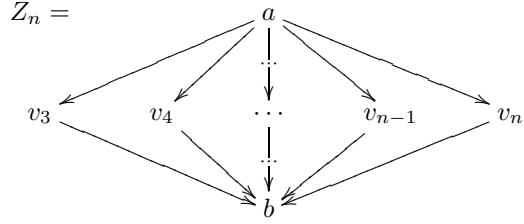
\[\xymatrix{Z_n= & &a\ar[dll]\ar[dl]\ar[d]|\cdots\ar[dr]\ar[drr]\\
v_3\ar[drr]&v_4\ar[dr]&\cdots\ar[d]|\cdots&v_{n-1}\ar[dl]&v_n\ar[dll]\\
&&b}\]
We can get a complete description of the elements in $HK_{Z_n}$ by observing that, since $a$ is a source, any element contains a word with at most one $a$. The subwords before and after $a$ do not contain $a$ and could therefore be thought of as living in $\big(V(Z_n\setminus\{a\})\big)^*$. However, in $Z_n\setminus\{a\}$ every vertex is a sink or a source, so we may assume that there is at most one $v_i$ on each side of $a$. A similar argument holds for $b$. Thus  each element $[\w]$ in $HK_{Z_n}$ contains a word $\w'$ which falls into exactly one of the following types:
\begin{enumerate}
\item We have $a,b\not\in\mathfrak{c}(\w')$ and $\w'$ multiplicity free. Since all $v_i$ commute with each other we may take $\w'=v_{i_1}v_{i_1}\cdots v_{i_j}$ with $i_1<i_2<\cdots<i_j$.
\item We have $a\in\mathfrak{c}[\w'],b\not\in\mathfrak{c}[\w']$ and $\w'=\w_1a\w_2$ is multiplicity free. Similarly as in the first case, $\w_1,\w_2$ may be taken with internally increasing order.
\item We have $a\not\in\mathfrak{c}[\w'],b\in\mathfrak{c}[\w']$ and $\w'=\w_1b\w_2$ is multiplicity free. As in 
the second case, $\w_1,\w_2$ may be taken with internally increasing order.
\item\label{eq71} We have $a,b\in\mathfrak{c}[\w']$ and $\w'=\w_1a\w_2b\w_3$. We may take $\w'$ such that 
$\mathfrak{c}(\w_1)\cap\mathfrak{c}(\w_2)=\emptyset$ and $\mathfrak{c}(\w_2)\cap\mathfrak{c}(\w_3)=\emptyset$
and $\w_1,\w_2$ and $\w_3$ are internally increasing.
\item\label{eq72} We have $a,b\in\mathfrak{c}[\w']$ and  $\w'=\w_1b\w_2a\w_3$ with the same restrictions as in the previous case
and, additionally, $\w_2\neq\varepsilon$. 
\end{enumerate}
The proof of the next theorem shows that all these elements are, in fact, different.

\begin{thm}
Let $\G=Z_n$ and let $f$ be defined by $f_{av_i}=1,f_{v_ib}=2^i$ for all $i$. 
Then $R_f$ is an effective representation of $HK_{Z_n}$.
\end{thm}

\begin{proof}
We prove the claim by induction on $n$ in the claim ``the elements of $Z_n$ are described as above and $R_f$ is an effective representation of $HK_{Z_n}$'' for $n\geq 3$. The base case, $n=3$, is obvious. Assume that the claim is true for some $n-1\geq 3$. Now we want to prove the claim for $n$.

First we note that $R_f([\w])(x)=x$ if and only if $x\not\in\mathfrak{c}(\w)$. This means that $R_f$ distinguishes
elements with different contents from each other. Assume first that the content of  $\w$ satisfies
$a\not\in \mathfrak{c}(\w)$ or $b\not\in \mathfrak{c}(\w)$. Then $[\w]$ belongs to some
$HK_{\Gamma'}$ for which $\Gamma'$ is a path connected subgraph of $\G$ and is a finite edge-disjoint union 
of graphs of type $A_2$. Since 
$\G$ has no oriented cycles, there is an $HK_{\Gamma'}$-sub\-quotient of $R_f$ isomorphic, as a vector space,
to $\oplus_{v\in V(\G')}\mathcal{R}v$ with the induced action (i.e. everything outside this space is treated
as zero). From the results of the previous section it follows that the restriction of $R_f$ to the action of
$HK_{\Gamma'}$ on $\oplus_{v\in V(\G')}\mathcal{R}v$ is effective and hence separates different elements with content 
$\mathfrak{c}(\w)$.

If $a,b\in\mathfrak{c}(\w)$ but some $v_i\not\in \mathfrak{c}(\w)$. Without loss of generality assume that $i=n$. This brings us to $Z_{n-1}$, which by the induction assumption is fine.

Finally, it remains to show that $R_f$ separates elements of the maximal content, that is elements described
in the last two cases \eqref{eq71} and \eqref{eq72} before the theorem with the condition
$\mathfrak{c}(\w_1)\cup\mathfrak{c}(\w_2)\cup\mathfrak{c}(\w_3)=\{v_3,v_4,\dots,v_n\}$.

A direct computation shows that the element $\w_1a\w_2b\w_3$ from \eqref{eq71} acts on the basis elements of 
$\oplus_{v\in V(\G)}\mathcal{R}v$ as follows:
\begin{gather*}
a\mapsto 0;\quad\quad b\mapsto \sum_{v_i\not\in \mathfrak{c}(\w_1)}f_{av_i}v_i+
\sum_{v_i\in \mathfrak{c}(\w_1)}f_{av_i}f_{v_ib}a;\\ 
\mathfrak{c}(\w_3)\ni v_i\mapsto 0;\quad\quad 
\mathfrak{c}(\w_2)\ni v_i\mapsto f_{v_ib}b;\quad\quad 
\mathfrak{c}(\w_1)\setminus \mathfrak{c}(\w_3)\ni v_i\mapsto f_{v_ib}b.
\end{gather*}
Similarly, the element $\w_1b\w_2a\w_3$ from \eqref{eq72} acts as follows:
\begin{gather*}
a\mapsto 0;\quad\quad b\mapsto \sum_{v_i\not\in \mathfrak{c}(\w_1)\setminus \mathfrak{c}(\w_2)}f_{av_i}v_i+
\sum_{v_i\in \mathfrak{c}(\w_1)\setminus \mathfrak{c}(\w_2)}f_{av_i}f_{v_ib}a;\\ 
\mathfrak{c}(\w_3)\ni v_i\mapsto 0;\quad\quad 
\mathfrak{c}(\w_2)\ni v_i\mapsto 0;\quad\quad 
\mathfrak{c}(\w_1)\setminus \mathfrak{c}(\w_3)\ni v_i\mapsto f_{v_ib}b.
\end{gather*}
Comparing these two formulae it is easy to see that different words act differently, in particular, they 
define different elements in $HK_{Z_n}$. 
\end{proof}
	
We can describe the graphs whose HK-monoids we have constructed effective representations for. We do it in three steps. 
\begin{enumerate}
\item Pick a forest $\G'$ whose edges are undirected.
\item Direct every edge such that every vertex in $\G'$ becomes either a source or a sink. Call the resulting graph $\G''$
\item Replace every edge in $\G''$ independently with one of the following:
\begin{enumerate}
	\item a linearly ordered graph of type $A_n,n\geq 1$,
	\item a "Kiselman graph" $\kappa_n,n\geq 2$,
	\item or some $Z_n$ for $n\geq 4$,
\end{enumerate}
matching source to source and sink to sink. Call the resulting graph $\G$.
\end{enumerate}	
By choosing any (integer) weights on the linearly ordered parts, weights according to \cite{KM09} on the "Kiselman graph"-parts and weights according to the previous theorem on the $Z_n$-parts one obtains an effective representation of $HK_\G$ over any ground ring $\mathcal{R}\supset\mathbb{Z}$.

\begin{exmp}
		A graph that we can reach with the three step algorithm
		\begin{enumerate}
		\item Pick a forest:
		
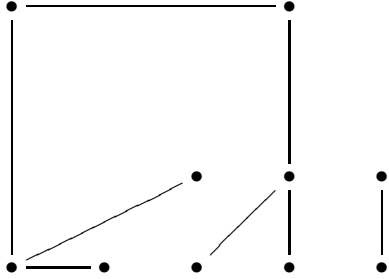
\[\xymatrix{\bullet\ar@{-}[rrr]\ar@{-}[ddd]&&&\bullet\ar@{-}[dd]\\
		\\
		&&\bullet\ar@{-}[dll]&\bullet\ar@{-}[dl]\ar@{-}[d]&\bullet\ar@{-}[d]\\
		\bullet\ar@{-}[r]&\bullet&\bullet&\bullet&\bullet}\]
		\item Direct the edges so that vertices become \emph{only} sources and sinks:
		
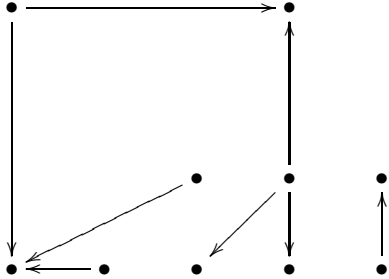
\[\xymatrix{\bullet\ar[rrr]\ar[ddd]&&&\bullet\ar@{<-}[dd]\\
		\\
		&&\bullet\ar[dll]&\bullet\ar[dl]\ar[d]&\bullet\ar@{<-}[d]\\
		\bullet\ar@{<-}[r]&\bullet&\bullet&\bullet&\bullet}\]
		\item Replace edges with linearly ordered $A_n$-graphs, graphs $\kappa_n$ or graphs $Z_n$:
		\[\xymatrix{&\\ \bullet\ar[d]\ar[r]\ar@(ur,ul)[rr]\ar@(u,u)[rrr]&\bullet\ar[r]\ar@(ur,ul)[rr]&\bullet\ar[r]&%
					\bullet\ar@{<-}[dll]\ar@{<-}[dl]\ar@{<-}[d]\ar@{<-}[dr]\\
		\bullet\ar[d]&\bullet\ar@{<-}[drr]&\bullet\ar@{<-}[dr]&\bullet\ar@{<-}[d]&\bullet\ar@{<-}[dl]\\
		\bullet\ar[d]&\bullet\ar[dl]&\bullet\ar[l]\ar[dll]&\bullet&\bullet\\
		\bullet&\bullet\ar[l]&\bullet\ar@{<-}[ur]&\bullet\ar@{<-}[u]&\bullet\ar[u]
		}\]
		\end{enumerate}
\end{exmp}
	

\end{document}